\documentclass[11pt, a4paper]{article}
\usepackage{amsmath,amscd}
\usepackage{amstext}
\usepackage{amssymb,epsfig}
\usepackage{amsthm}
\usepackage{amsfonts}
\usepackage{graphicx}
\usepackage{psfrag}
\usepackage[all]{xy}
\usepackage{tikz} 
\usetikzlibrary{arrows,automata}

\renewcommand{\footnote}{\endnote}

\newtheorem{theorem}{Theorem}

\newtheorem{lemma}{Lemma}
\newtheorem{proposition}{Proposition}
\newtheorem{corollary}{Corollary}

\theoremstyle{definition}
\newtheorem{definition}{Definition}
\newtheorem{remark}{Remark}

\usepackage[utf8]{inputenc}      % pour les accents

\begin{document}

\title{Language of a non-minimal billiard trajectory inside a cube}

\author{Moussa Barro \footnote{Universit\'e Nazi BONI, 01 BP 1091, Bobo-Dioulasso, Burkina Faso. Email: mous.barro@yahoo.com}\and Nicolas B\'edaride\footnote{ Aix Marseille Université, CNRS, I2M UMR 7373, 13453 Marseille, France. Email: nicolas.bedaride@univ-amu.fr}\and Julien Cassaigne\footnote{  CNRS, I2M UMR 7373, Aix Marseille université, 13453 Marseille, France. Email: julien.cassaigne@math.cnrs.fr}}
\date{}

\maketitle

\begin{abstract}
We consider a non-minimal billiard trajectory inside the cube. We study the language of the associated orbit when the map is coded with three letters associated to three non-parallel faces of the cube.
\end{abstract}

%\tableofcontents

%%%%%%%%%%%%%%%
%%%%%%%%%%%%%%%
%%%%%%%%%%%%%%%
\section{Introduction}

This paper deals with billiard inside the unit cube $P=[0,1]^3\subset \mathbb{R}^3$. A billiard ball, i.e.\ a point mass, moves inside a cube $P$
with unit speed along a straight line until it reaches the
boundary $\partial{P}$, then it instantaneously changes direction
according to the mirror law, and continues along the new line. We want to define and study subshifts associated to the billiard map.

Label the faces of $P$ by symbols from a finite alphabet $\mathcal{A}=\{a,b,c\}$ (opposite faces are labeled by the same letter). For a direction $\theta$ in $\mathbb R^3\setminus\{0\}$ and a point $m$ in $\partial{P}$, the orbit of $(m,\theta)$, if it exists, can be coded in an infinite word, which defines a subshift of $\mathcal{A}^{\mathbb N}$, denoted $X_{m,\theta}$. We can also consider the  subshift $X_\theta$ which is the adherence of $\bigcup_{m}X_{m,\theta}$. Finally we can consider the subshift $X$, adherence of $\bigcup_{\theta}X_{\theta}$. Of course we have
$$X_{m,\theta}\subseteq X_{\theta}\subseteq X.$$

In all the cases, the orbit of a point corresponds to a (infinite) word on the alphabet $\mathcal{A}$. In other words, the first subshift consider words obtained in the orbit of one point $(m,\theta)$, the second subshift consider all words obtained in the direction $\theta$, and the last one consider all billiard words.

A subshift is thus a collection of infinite words, and it has a language which is the collection of finite words which appear in one of these words. We define the complexity of a subshift, as the function $n\mapsto p(n)$, by the number of words of length $n$ that appears in this language. We have three functions $p(n,m,\theta), p(n,\theta)$ and $p(n)$ for our three subshifts and of course $p(n,m,\theta)\leq p(n,\theta)\leq p(n)$.

A direction $\theta$ is said to be minimal if the closure of a billiard orbit in this direction passes through almost all points of $P$. For the cube, it is equivalent to the fact that the coordinates of $\theta$ are linearly independant over $\mathbb{Q}$, see \cite{tab}. In this case the subshift $X_\theta$ is minimal. Under this hypothesis $p(n,m,\theta)=p(n,\theta)$ for almost every point $m$.

In the two dimensional case, the same problems were studied, in a fixed direction, for rational polygons in \cite{Hubert}. For the global complexity (the language of all directions together) the growth of the complexity was obtained in \cite{Cassaigne-Hubert-troub} and more precize results are in \cite{Athreya-Hubert-troub} for regular n-gones.

The directional complexity in the cube, for minimal directions, has been computed by Arnoux, Mauduit, Shiokawa and Tamura in \cite{Ar.Ma.Sh.Ta.94} and generalized to the hypercube by Baryshnikov \cite{Ba.95}. There was a mistake in the first paper in the hypothesis on the direction, see \cite{Bed.07} and \cite{Bed03} for the precise definition of the good hypothesis which is more restrictive than being minimal. 
%Under this stengthened hypothesis. 
If $\theta$ is a minimal direction but does not fulfill the hypothesis, then $p(n,\theta)\sim C_\theta n^2$ with $0<C_\theta<1$. If the direction is not minimal then $p(n,\theta)\leq Cn$ for $n\geq 1$, see \cite{Bed.07}. The aim of this paper is to investigate the subshift associated to non-minimal directions and obtain a better understanding. 

In all the folllowing, we study all the trajectories of direction $\theta_0=\begin{pmatrix}
\frac{1}{2}& 
\frac{1}{\varphi}& \frac{1}{\varphi^2}
\end{pmatrix}$, where $\varphi$ is the golden mean $\frac{1+\sqrt 5}{2}$. We choose this direction so that one projection on a face of the cube gives a trajectory coded by a Sturmian word with the same slope as the Fibonacci word: $\frac{1}{\varphi}$, and so that the letter $c$ has frequency $\frac{1}{3}$. The results can be generalized: they will be of the same nature if we replace $1/2$ with $r\in\mathbb Q$, moreover in the last Section we explain how to obtain the same type of results on other directions.

\medskip

It appears that this problem is related to another one from combinatorics on words. In \cite{Bar.Kab.17} the authors considered insertion of a letter (from the alphabet or not) inside a Sturmian word, and they studied the combinatorial properties of the obtained word. We will link the two problems and show how to deduce some new results. For other related results on billiard and combinatorial on words we can also refer to \cite{Frid-Jamet} and \cite{Akiyama} for results on complexity function, or to \cite{Andrieu-vivion} for results on the language of a billiard word.

\section{Codings of dynamical systems}\label{sec-coding}
We give a definition of the coding of a billiard trajectory in the cube. This definition is not the usual one, but in Section \ref{sec-billiard} we will link it with the usual definition.

\begin{definition}
Consider a direction $\theta\in \mathbb R_+^3\setminus\{0\}$ and a point $m\in \mathbb R^3$. Assume that the line $m+\mathbb R\theta$ does not contain a point with more than one integer coordinate (otherwise the trajectory is not defined). Then consider the sequence of intersections of the half line $m+\mathbb R_+\theta$ with the cellular complex $\mathbb Z^3$.
It defines an infinite word $u$ on the alphabet $\{a,b,c\}$ with $a$ for the intersection with $X=n, n\in \mathbb Z$, $b$ for the intersection with a plane $Y=n, n\in \mathbb Z$, and $c$ for the intersection with $Z=n$, see Figure \ref{fig-coding}. We call $u$ a {\bf billiard word} (the meaning will be explained in the next section). The coding map of the billiard trajectory of $(m,\theta)$ is the map $f_{\theta}$ defined almost every where: $$\begin{array}{ccc}
P&\mapsto& \mathcal{A}^{\mathbb N}\\
m&\mapsto&f_\theta(m):=u\end{array}$$
\end{definition}

\begin{figure}
\begin{center}
\begin{tikzpicture}[scale=.5]
\draw (0,0)--(2,0)--(2,2)--(0,2)--cycle;
\draw (1,1) node{$b$}; 

\draw (5,1)--(7,1)--(8,2)--(6,2)--cycle;
\draw (6.5,1.5) node{$c$}; 

\draw (9,0)--(10,1)--(10,3)--(9,2)--cycle;
\draw (9.5,1.5) node{$a$};

\draw[->] (0,-2)--++(1,0);
\draw (1.4,-2) node{$X$};
\draw[->] (0,-2)--++(0,1);
\draw (0,-1) node[left]{$Z$};
\draw[->] (0,-2)--++(1,1);
\draw (1.5,-1) node{$Y$};
\end{tikzpicture}
\end{center}
\caption{Coding of the billiard inside the cube}\label{fig-coding}
\end{figure}
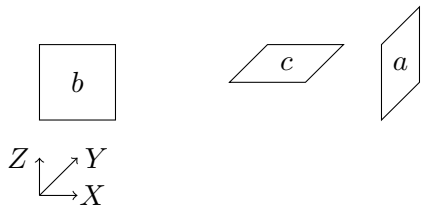
We denote by $\mathcal L_{m,\theta}$ the language of $u$ and by $\delta_a$ (resp. $\delta_b, \delta_c$) the erasing of a letter in a word, see Definition \ref{def-langage}.
Finally we denote by $\pi_x$ (resp.~$\pi_y, \pi_z$) the orthogonal projection on the plane $OYZ$ (resp. $OXZ, OXY$). The relation between these projections will be explained in Lemma \ref{proj}.

%%%%%%%%%%%%%%%%
%%%%%%%%%%%%%%%%%
\section{Result and overview of the paper}

 We refer to Section \ref{sec-coding} for the coding of the point $m$ in the direction $\theta$. The statement of the results involves several objects defined after: a new alphabet, a morphism $\Phi$ from the new alphabet to the old one (see Lemma \ref{lem-return-word}), 
a translation on the interval $[0,1]$, a coding map $g$ of this translation on this new alphabet (see Definition \ref{def-rot-coding}), and a map $y$ such that the image of the point $m$ in the face $X=0$ of the cube is a point $y(m)$ of $[0,1]$ (see Definition \ref{def-zm}).
The first theorem says that a billiard word is a recoding of a translation on $[0,1]$ coded on more than two intervals.

\begin{theorem}\label{thm1}
Let us consider the cube $P=[0,1]^3$ coded with three letters, and the direction $\theta_0$ defined previously. For almost all $m$ in the face $X=0$, there exists %an integer $k(m)\in \{3,5,6\}$, 
a partition of $[0,1]$ in  $6$ intervals $(I_i)_{1\leq i\leq 6}$ such that $f_{\theta_0}(m)=\Phi(v)$ where $v=g(y(m))$. %
\end{theorem}

The second theorem gives precise values for the complexity functions. It uses the first theorem in the proof.
\begin{theorem}\label{thm2}
$\;$

\begin{itemize}
\item For all $m=(0,y,z)$ in the cube which has a well defined billiard orbit, the complexity of $\mathcal L_{m,\theta_0}$ is linear and only depends on $y+z$. 
\item There are $6$ values for the complexity function of $\mathcal L(m,\theta_0)$ when $m$ varies in the face.
\item For almost all $m$ in the face $X=0$, we have
$ p(n,m,\theta_0)=4n-1$ for every large enough integer $n$.
\end{itemize}
\end{theorem}

The last theorem gives an equivalent to the complexity of $X_{\theta_0}$.
\begin{theorem}\label{thm3}
The complexity function of the language $\mathcal L_{\theta_0}$ has the following asymptotic property: $$p(n,\theta_0)\sim \frac{4+\varphi}{6}n^2.$$

\end{theorem}

%\bibliographystyle{abbrv}
%
%\bibliography{biblio-bbc.bib}

%\end{document}
%%%%%%%%%%%%%
%%%%%%%%%%
\section{Background}
\subsection{Combinatorics}
\begin{definition}\label{def-langage}
Let $\mathcal{A}$ be a finite set called the alphabet. By a
language $\mathcal{L}$ over $\mathcal{A}$ we always mean a factorial
extendable language: a language is a collection of sets
$(\mathcal{L}_n)_{n\geq 0}$ where the only element of $\mathcal{L}_0$ is the empty
word, and each $\mathcal{L}_n$ consists of words of the form $a_1a_2\dots
a_n$ where $a_i\in\mathcal{A}$ and such that for each $v\in \mathcal{L}_n$
there exist $a,b\in\mathcal{A}$ with $av,vb\in \mathcal{L}_{n+1}$, and for
all $v\in \mathcal{L}_{n+1}$ if $v=au=u'b$ with $a,b\in\mathcal{A}$ then
$u,u'\in \mathcal{L}_n$.\\
The complexity function $p:\mathbb{N}\rightarrow\mathbb{N}$ is
defined by $p(n)=card(\mathcal{L}_n)$.
\end{definition}

 First we recall a well-known result of Cassaigne
concerning combinatorics of words \cite{Cass.97}.
\begin{definition}
Let $\mathcal{L}(n)$ be a factorial extendable language. For $v \in \mathcal{L}(n)$ let
$m_{l}(v)= card\{a\in \mathcal A,av\in \mathcal{L}(n+1)\},$ $m_{r}(v)= card\{b\in \mathcal A,vb\in \mathcal{L}(n+1)\},$ and
$m_{b}(v)= card\{a\in \mathcal A, b\in \mathcal A, avb\in, \mathcal{L}(n+2)\}.$ Now let
$i(v)=m_{b}(v)-m_{r}(v)-m_{l}(v)+1.$

A word $v$ is called right special if $m_{r}(v)\geq 2$, it is called left special if
$m_{l}(v)\geq 2$ and it is called bispecial if it is both right and left special.
Let $\mathcal{BL}(n)$ be the set of bispecial words.
\end{definition}

Cassaigne \cite{Cass.97} has shown:
\begin{lemma}\label{julien}
Let $\mathcal{L}$ be a language. For any
$n\geq 0$ let $s(n)\!:=p(n+1)-p(n)$. Then the complexity satisfies
$$\forall n\geq 0 \quad s(n+1)-s(n)=\sum_{v\in \mathcal{BL}(n)}{i(v)},$$

\end{lemma}

\subsection{Billiard}\label{sec-billiard}
Let $P$ be the unit cube, the billiard map is called $T$ and it is defined on a
subset of $\partial{P}\times \mathbb{R}\mathbb{P}^{2}$. This space is called the
phase space.

The unfolding of a billiard trajectory is an important tool that is used: Instead of
reflecting the trajectory in the face we reflect the cube and
follow the straight line. Thus we consider the tiling of
$\mathbb{R}^{3}$ by $\mathbb{Z}^{3}$, and the associated partition
into cubes with edges of length one. In the following when we use the term "face" we mean 
a face of one of those cubes.

Now fix a direction $\theta$ and a point $m$ in the boundary of $P$. The billiard orbit of $(m,\theta)$ is then coded by the infinite word $u=f_\theta(m)$ on a three-letter alphabet, as explained in Section \ref{sec-coding}.

The following lemma is very useful,  its proof is easy and left to the reader.

\begin{lemma}\label{proj}
Consider an orthogonal projection on a face of the cube. The orthogonal projection of a billiard map is a billiard map inside a square. If $u$ codes one trajectory $t$, then $\delta_a(u)$ codes the trajectory $\pi_x(t)$, and same thing for the other projections.

In a square coded with two letters, the coding of a non-periodic billiard trajectory is a Sturmian word, while the coding of a periodic billiard trajectory is a periodic word.
\end{lemma}

We deduce the following corollary

\begin{corollary}\label{cor-proj-sturm}
Let $\theta=\begin{pmatrix}
\theta_1&\theta_2&\theta_3
\end{pmatrix}\neq 0$ and $m$ be a point on the boundary of the cube. Let $u$ be the coding of the orbit of $(m,\theta)$. Then $\delta_a(u)$ is a Sturmian word if and only if $\begin{pmatrix}\theta_2&\theta_3\end{pmatrix}$ are rationally independant over $\mathbb Q$.
\end{corollary}
\begin{proof}
Consider the orthogonal projection on the plane $(OYZ)$. By Lemma \ref{proj} the projection of the orbit of $(m,\theta)$ is a billiard orbit inside a square. The direction associated to it is $\begin{pmatrix}
\theta_2& \theta_3
\end{pmatrix}$. Thus it is non-periodic if and only if $\theta_2, \theta_3$ are rationally independant over $\mathbb Q$. And it is clear that the coding of the orbit inside the square is $\pi_a(u)$.
\end{proof}

\begin{remark}
Now, in the following we will only consider direction $\theta\in \mathbb R_+^3\neq\{0\}$ such that $\theta_i\geq 0$ for all $i=1,2,3$. Up to symmetry we can always reduce to this case.
\end{remark}

We refer to \cite{tab} for the proof of the following classical fact.
\begin{lemma}
Consider a billiard trajectory in direction $\theta$, then the frequency of the letter $a$ is equal to $\frac{\theta_1}{\theta_1+\theta_2+\theta_3}$. The same is true up to permutation for the other letters. 
\end{lemma}

Now if $\frac{\theta_3}{\theta_1+\theta_2+\theta_3}$ is equal to $1/k, k\geq 1 \in \mathbb N$, then the frequency of the letter $c$ in $u$ is equal to $1/k$. In a very particular way, it is possible that this letter $c$ appears each $k-1$ times in $u$. This was exactly the case treated in \cite{Bar.Kab.17}.

Thus we can rephrase  our problem:
Starting from a sturmian word $v$ on the alphabet $\{b,c\}$, we are looking to a word $u$ on the alphabet $\{a,b,c\}$ such that $\pi_a(u)=v$ and the frequency of $a$  in $u$ is fixed to $1/3$.

Starting from a square billiard orbit in a direction $\begin{pmatrix}
\theta_2& \theta_3
\end{pmatrix}$ with $\theta_2/\theta_3\notin \mathbb Q$ we are looking, due to Corollary \ref{cor-proj-sturm}, to a cubic billiard orbit in a direction $\begin{pmatrix}\theta_1&\theta_2& \theta_3\end{pmatrix}$ which projects on the word $v$ by $\pi_a$ and such that $\frac{\theta_1}{\theta_1+\theta_2+\theta_3}=\frac{1}{3}$. 
This explains the choice of $\theta_0$.
%%%%%%%%%%%%%
%%%%%%%%%%%%%%
\section{Lemmas and proof of Theorem \ref{thm1}} 

\subsection{Lemmas}

Consider $m$ in the face $X=0$ of the cube, then the orbit of $(m,\theta_0)$ is coded by $u$ which begins with the letter $a$. Moreover we know that $\pi_a(u)$ is the Fibonacci word in the alphabet $\{b, c\}$.  
Thus to describe the language it suffices to know the positions of the occurences of $a$ in $u$.
We call the words, factors of $u$, which begins by $a$ and finish before the next $a$, a return word of $a$, following the work of Durand. We have the following expression
$$u=a\dots a\dots a\dots$$

We denote by the first return word as the first of the return words of $a$.
\begin{figure}
\begin{center}
\begin{tikzpicture}[scale=1.3]
\draw (0,0)--(4,0)--(4,4)--(0,4)--cycle;
\draw[blue] (0,1)--++(4,0);
\draw[blue] (3,0)--++(0,4);
\draw[red] (0,1.6)--(4,4);
\draw[blue] (3,1)--(4,1.6);

\draw (0,0) node[left]{$0$};
\draw (4,0) node[right]{$1$};
\draw (1,3) node{$P_{a_1}$};
\draw (2,1.5) node{$P_{a_2}$};
\draw (3.5,2) node{$P_{a_3}$};
\draw (3.5,.5) node{$P_{a_4}$};
\draw (3.7,1.2) node{$P_{a_5}$};
\draw (3.2,3.85) node{$P_{a_6}$};
\draw (2,.5) node{$P_{a_7}$};

%\draw (0,3) node[left]{$\frac{1}{\varphi}$};
\draw (-.6,1) node[below]{$1-\frac{2}{\varphi^2}=2\varphi-3$};
\draw (3,0) node[below]{$2-\frac{2}{\varphi}=4-2\varphi$};
\draw (4,1.6) node[right]{$1-\frac{1}{\varphi}$};
\draw (0,1.6) node[left]{$1-\frac{1}{\varphi}=2-\varphi$};
\end{tikzpicture}
\end{center}
\caption{Return words and partition of the face $X=0$.}\label{fig-partition}
\end{figure}
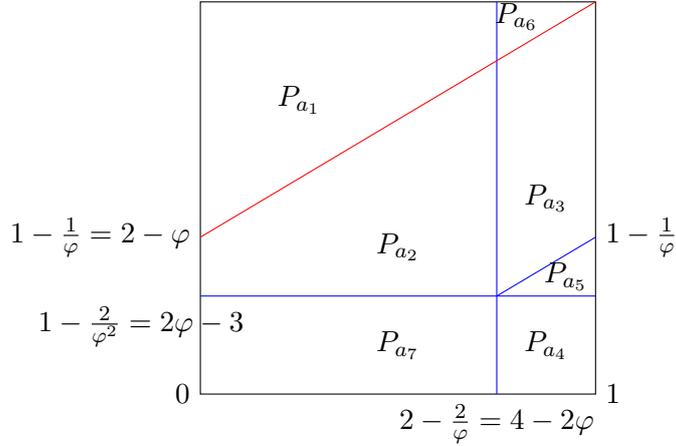

\begin{lemma}\label{lem-return-word}
There exists a partition $\mathcal{P}=\bigcup_{1\leq i\leq 7} P_{a_i}$ of the face $X=0$ of the cube and a map $\Phi$ from $\{a_1,\dots,a_7\}$ to $\mathcal A^*$ such that $\Phi(x)$ is the first return word of $a$ in $f_{\theta}(m)$ if $m\in P_x$.  

$$\begin{array}{c|c|c|c|c|c|c|c|}
m\in P_{a_i}&a_1&a_2&a_3&a_4&a_5&a_6&a_7\\
\hline
\Phi(a_i)&acb&abc&abcb&abb&abbc&acbb&ab
\end{array}$$
\end{lemma}
\begin{proof}
We will refer to Figure \ref{fig-partition}.
consider $m$ in the face labeled by $a$. The coding of the orbit of $m$ contains the letter $a$ if and only if the half line of direction $\theta$ intersects a face $X=n$ for $n\in \mathbb N$. Thus the list of return words of $a$ in $u$ is given by the intersections with this faces. 
Let us consider the projection on the plane $X=0$ along $\mathbb R\theta$. 
\begin{itemize}
\item The line $X=n, Y=m, (n,m\in\mathbb N)$ has for image the line of equation $Y=m-\frac{n\theta_2}{\theta_1}$
\item The line $X=n, Z=m, (n,m\in\mathbb N)$ has for image the line of equation $Z=m-\frac{n\theta_3}{\theta_1}$
\item The line $Y=n, Z=m, (n,m\in\mathbb N)$ has for image the line $\theta_3 Y-\theta_2Z=n\theta_3-m\theta_2$.
\end{itemize}

Remark that $\theta_2=\frac{1}{\varphi},\theta_3=\frac{1}{\varphi^2}$ and $\varphi<2<\varphi^2$. Thus on the half line $\mathbb R_+ \theta$, the first point with one integer coordinate is $(\varphi/2,1,1/\varphi)$. Thus  the first of the base planes intersected by the  half line $\mathbb R_+\theta$ is $Y=1$. Thus $aa$ does not exist in the language of a billiard word in direction $\theta$. 

The image of the line $Z=Y=1$ by the projection is $\theta_3 Y-\theta_2Z=\theta_3-\theta_2$. Thus the segment split the face $X=0$ in two parts: one associated to the word $ab$ and one to the word $ac$. Remark that on Figure \ref{fig-partition} it corresponds to the red segment. 

The letter $c$ is followed by a letter $b$ or a letter $a$. Indeed consider the projection on the face $(OYZ)$ and remark that the slope of the line is less than $1$.
Thus some return words begin with $ab$ and the others with $acb$.
The projection of a line parallel to $(OZ)$ is also parallel to $(OZ)$, thus the blue segment in Figure \ref{fig-partition} correspond to the letter $b$.

The last thing to remark is that one diagonal of the square corresponds to the projection of a segment parallel to $Z=Y=1$.

\end{proof}

\begin{remark}
We will do an abuse of notation and denote $\Phi(m)=\Phi(a_i)$ if $m\in P_{a_i}$.
\end{remark}
Now we must understand how the return words are combined inside $u$.

\begin{lemma}\label{lem-translation}
Let $m$ be in the face $X=0$ of the unit cube, then the $(k+1)$-th return word in the billiard word $u=f_\theta(m)$ is $\Phi(m+2k\theta_2\begin{pmatrix} 1& -1\end{pmatrix} \mod 1)$. 
\end{lemma}
\begin{proof}
Let $m$ be in the face $X=0$, then the $(1+k)$-th return word appears after the $k+1$ occurence of the letter $a$, thus after the half-line $m+\mathbb R \theta$ intersects $X=k$. Let us denote  $m_0:=m+\lambda_0 \theta$ the point in $X=0$, and $m_{k}:=m+\lambda_{k} \theta$ in $X=k$. Thus we deduce that $(\lambda_{k}-\lambda_0)\theta$ has for first coordinate $k$, thus $\lambda_{k}-\lambda_0=2k$ and $m_k=m_0+2k\theta$. If we fold the trajectory, it means that in the face $X=0$, we have two points $p_0, p_k$ and 
$p_k-p_0=2k\begin{pmatrix}
\theta_2& \theta_3
\end{pmatrix} \mod 1.$ But $\theta_3=1-\theta_2$, thus 

$$p_k-p_0=2k\begin{pmatrix}
\theta_2& 1-\theta_2
\end{pmatrix}=2k\theta_2\begin{pmatrix}1& -1\end{pmatrix} \mod 1$$

The coding of the billiard trajectory between $m$ and $p_k$ corresponds to the concatenation of $k$ return words.
\end{proof}

In other words, the action of $2\theta_2\begin{pmatrix} 1& -1\end{pmatrix} \mod 1$ is coded by $\Phi^{-1}(u)$. 

\begin{figure}

\begin{tikzpicture}[scale=1]
\draw (0,0)--(4,0)--(4,4)--(0,4)--cycle;
\draw (0,1)--++(4,0);
\draw (3,0)--++(0,4);
\draw (0,1.6)--(4,4);
\draw (3,1)--(4,1.6);

\draw[dashed] (0,4)--(4,0);
\draw[ green] (0,1)--(1,0);
\draw[ green] (1,4)--(4,1);

\draw[blue] (0,1.6)--(1.6,0);
\draw[blue] (1.6,4)--(4,1.6);

\draw[red] (3,3.4)--(4,2.4);
\draw[red] (0,2.4)--(2.4,0);
\draw[red] (2.4,4)--(4,2.4);

\draw[pink] (3,4)--(4,3);
\draw[pink] (0,3)--(3,0);

\draw[blue] (6,0) circle(1);
\fill (6,1) circle(.1); 
\fill (7,0) circle(.1); 
\fill (5.3,-.7) circle(.1); 
\fill (5.3,.7) circle(.1); 
\fill (6.2,-.95) circle(.1); 

\draw[dotted] (6,3) circle(1);
\fill (8.5,2.1) circle(.1); 
\fill (10,3) circle(.1); 
\fill (8.5,3.85) circle(.1); 
\fill (9.5,3.85) circle(.1); 
\fill (9,4) circle(.1); 

\draw[pink] (9,3) circle(1);
\draw[green] (9,0) circle(1);
\fill (7,3) circle(.1); 
\fill (6,4) circle(.1); 
\fill (5.5,2.1) circle(.1); 

\draw[red] (12,2) circle(1);
\end{tikzpicture}

\caption{Five circles on the torus with different partitions. One in dot corresponds to three intervals, and the other correspond to five intervals.}\label{fig-lineaire}
\end{figure}
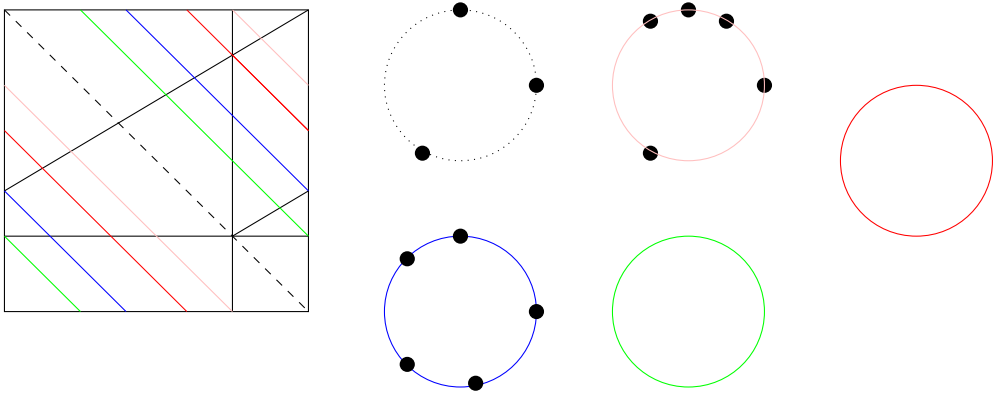

\begin{figure}
\begin{tikzpicture}
\draw (0,0)--(4,0)--(4,4)--(0,4)--cycle;
\draw (0,1)--++(4,0);
\draw (3,0)--++(0,4);
\draw (0,1.6)--(4,4);
\draw (3,1)--(4,1.6);

\draw[yellow] (0,.5)--(.5,0);
\draw[yellow] (.5,4)--(4,.5);
\end{tikzpicture}
\caption{A circle on the torus with a partition in six intervals.}\label{fig-partition6}
\end{figure}
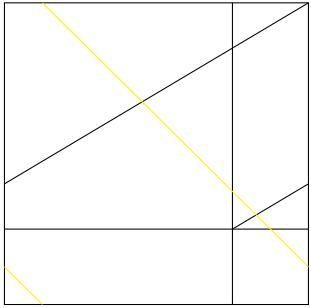

Consider the face $X=0$ of the cube, this square is a fundamental domain of a torus $\mathbb T^2$.

\begin{lemma}\label{lem-nb-intervalles}
Consider the linear flow by $\begin{pmatrix}
1&-1\end{pmatrix}$ on the torus $\mathbb T^2$. Let $m$ be a point in this torus. Then the orbit of this point under the linear flow is a circle. Its length does not depend on $m$.

The orbit of $m$ under the flow intersects $k(m)$ elements of the partition $\mathcal{P}$, with $k(m)\in \{3,5,6\}$. Each of these intersections is an interval.
\end{lemma}
We will denote by $I$ this orbit.
\begin{proof}
The linear flow by $\begin{pmatrix}
1&-1\end{pmatrix}$ is periodic since its slope is a rational number. Thus $I$ is a circle, that we can identify with an interval if we assume that the origin is the point on $X=0$.

Starting with $m=\begin{pmatrix}0&0&z\end{pmatrix}$ we see that the intersection of the partition with the action of the translation has $k(m)$ pieces, and we have, see Figure \ref{fig-lineaire} and Figure \ref{fig-partition6}:
\begin{itemize}
\item If $z=0$, then $k(m)=3$, see the dashed circle.
\item If $z\in\{2\varphi-3, \varphi-1, 2-\varphi,4-2\varphi\}$ then $k(m)=5$, see the pink, green, red, blue circles.
\item Otherwise $k(m)=6$.
\end{itemize}

\end{proof}

\begin{definition}\label{def-zm}
We denote $y(m)$ the second coordinate of $m$. With previous result it gives the position of $m$ on this circle. 
\end{definition}

\begin{definition}\label{def-rotation}
We denote by $R$ the following translation on $[0,1]$ $R:y\mapsto y+\frac{2}{\varphi} \mod 1$.
\end{definition}
We recall that this map is related to the translation by the same amount on the torus $\mathbb T^1$, if we use $[0,1]$ as fundamental domain of the torus.

Fix $m$ a point of the face $X=0$ of the cube, and let $k(m)\in \{3,5,6\}$ see Lemma \ref{lem-nb-intervalles}.

\begin{definition}\label{def-rot-coding}
Consider the partition of $I$ (seen as the interval $[0,1]$) in intervals obtained with the intersections of $I$ with the partition $\mathcal{P}$. 
Then we define the coding map $g$ of $R$ as 

$$
\begin{array}{ccc}
 [0,1]&\mapsto& \{a_1, \dots, a_7\}^{\mathbb N}\\
y&\rightarrow&g(y)\\
\end{array}$$

where $g(y)_n=a_i$ if and only if $R^n(y)\in I_j$ and there exists a unique $a_i$ such that $I_j\subset X_{a_i}$.
\end{definition}

\subsection{Proof of Theorem \ref{thm1}}

The coding $u$ of the orbit of $m$ starts with the letter $a$. It suffices to know the positions of $a$. 
By Lemma \ref{lem-return-word}, these positions are determined by the return words, and these words are concatenated in the order given by Lemma \ref{lem-translation}. 
Thus we have a map with a coding, which produces $u$.
Thus we obtain $f_\theta(m)=\Phi(g(y(m)))$ and
Lemma \ref{lem-nb-intervalles} gives the number of intervals

%%%%%%%%%%%%%%%%%%%%%%%%%%%%%
%%%%%%%%%%%%%%%%%%%%%%%%%%%%
\section{Proof of Theorem \ref{thm2}}

We refer to Figure \ref{fig-partition} and Figure \ref{fig-partition6}.

\subsection{First part of the proof}
Consider $m=\begin{pmatrix}
0&y&z\end{pmatrix}$ on the face $X=0$. Then we remark the following facts. By Lemma \ref{lem-translation} the language $\mathcal L_m$ depends only on the flow in the direction $(1 -1)$, thus it only depends on $y+z$. By Theorem \ref{thm1} the coding of the orbit of $m$ is the recoding of a translation. Next lemma will explain how to find the complexity.

\begin{lemma}\label{lem-rot-complexite}
Consider a translation of irrational angle $a$ on the torus $\mathbb T^1$. Consider a coding of this map on $d$ intervals $[a_i, a_{i+1}], 0\leq i\leq d-1$ with $a_0=0,    
 a_d=1$. Then there exists $n_0, A,B$ such that the complexity of almost every orbit of a point in $[0,1]$ is equal to $p(n)=An+B$ for $n\geq n_0$ for some $n_0\in \mathbb N$. Moreover $A$ is equal to the dimension over $\mathbb Z$ of the modulus generated by $a, a_1,\dots, a_{d-1}$.
\end{lemma}
\begin{proof}
A word $v$ in the language is represented by a segment (all points of the segment have an orbit whose coding begins with $v$ as prefix). This word is right special if the interval contains the backward orbit of some $a_i$. Indeed if $R^{-n}a_i$ is inside the interval then points on the left and right side of $R^{-n}a_i$ have orbits whose codings differ after $n$ steps. Thus we have to look if $a_i+na=a_j+ma$ for some integers $n,m$. The number of solutions of such equation is given by the dimension over $\mathbb Z$ of the modulus generated by $a, a_1,\dots, a_{d-1}$. Thus $p(n)$ is linear of the form $An+B$ for $n$ large enough.
\end{proof}

Now we state the following proposition which will prove the last point of the theorem

\begin{proposition}\label{prop-6int}
Consider $p=\begin{pmatrix}
0&0&z\end{pmatrix}$ with $z\in [0,1]$ a number in the complement of $\mathbb Q(\varphi)$, then
$\mathcal L'_m$ has complexity function $4n+2, n\geq 2$, and $\mathcal L_m$ has complexity function $4n-1, n\geq 1$.
\end{proposition}
\begin{proof}
We obtain a partition of the circle $I$ in $k(p)=6$ intervals. The discontinuity points are, see Figure \ref{fig-partition6}:
 $z,2-\varphi+(\varphi-1)z, 4-2\varphi, 4-2\varphi+(\varphi-1)z, 4-2\varphi+z$. 
The translation has an angle of $2\varphi-3$, see Definition \ref{def-rotation}. We see that four discontinuities have distinct orbits, due to the hypothesis on $z$. Thus the complexity of $\mathcal L_m'$ is equal to $p(n)=4n+2, n\geq 2$. Now we recode by $\Phi$ and by the same argument as before we obtain for all $n\geq 2$ four right special words of length $n$.
\end{proof}

\subsection{Second part}
By Lemma \ref{lem-nb-intervalles}, depending on $m$, the orbit of $m$ under the linear flow intersects different numbers of elements of the partition. In each of the cases we will apply Lemma \ref{lem-rot-complexite}.
\subsubsection{Translation on five intervals}

\begin{lemma}
We consider the four circles with a partition in five intervals. 
For almost every point $m
$ on each circle the complexity function $p(n,m,\theta_0)$ is equal to $2n+a$ for some $a$ for all $n$ large enough.

\end{lemma}
\begin{proof}
We refer to Figure \ref{fig-partition}, the computations are left to the reader. The lengths of these intervals are equal to (up to a factor $\sqrt 2$):
\begin{itemize}
\item Green circle: 
$2\varphi-3,10-6\varphi, 2\varphi-3,5-2\varphi, 4\varphi-8$.
\item Blue circle:
$5-3\varphi, 2\varphi-3, 2\varphi-3, 5-3\varphi, 2\varphi-3$.
\item Pink circle:
$5\varphi-8, 5-3\varphi, 2\varphi-3, 10-6\varphi, 2\varphi-3$.
\item Red circle $\varphi-1, 4-2\varphi,1, 5-3\varphi, 3-\varphi$.
\end{itemize}

Now we apply Lemma \ref{lem-rot-complexite} and deduce that in the four cases we have $p(n,m,\theta_0)\sim 2n$. 
\end{proof}

Remark that we could be more precize and compute the constant $a$. For example we have:
\begin{proposition}\label{Prop-5int}
Consider the point $\begin{pmatrix}
0&0&2-\varphi\end{pmatrix}$, then
$\mathcal L'_m$ has complexity function $2n+5, n\geq 2$, and $\mathcal L_m$ has complexity function $2n+8, n\geq 8$.
\end{proposition}
\begin{proof}
We have a partition of $I$ in $5$ intervals but two have the same label: indeed the labels are $a_2, a_7, a_1, a_2, a_3$ with discontinuity points at $5-3\varphi, 2-\varphi, \varphi-1, 4-2\varphi$. The map is a translation by $2\varphi -3$, see Definition \ref{def-rotation}. Thus the image of the partition is made by intervals $a_3,a_2,a_7, a_1,a_2$. We deduce that the word $a_2a_7a_1$ is a right special word with right extensions $a_2, a_3$, the word $a_2$ has right extensions $a_3, a_7$ and $a_1a_2a_3$ has extensions $a_2, a_7$.
Thus this language has for complexity function $p(1)=4$ and $p(n)=2n+5$ for $n\geq 2$.

Now we pass to $\mathcal L_m$ by the map $\Phi$. 
We deduce the three right special words $$abcabacbabc, cbabcab, acbabcabcbab.$$ Thus we have $s(n)=2$ for $n$ lare enough. Now it remains to look at small values of $n$:
$a,b,c$ are right special words of length one. Then $ba,ab, bc$ are right special words. Finally $bab, cab, abc$ are also right special words. %Pourquoi jusqu'à 8 ? \maltese 
We deduce $p(n)=2n+8$ for $n\geq 8$ and $p(1)=3, p(2)=6, p(3)=9, p(4)=12, p(5)=15, p(6)=18, p(7)=21.$
\end{proof}

\subsubsection{Translation on three intervals}
We consider the diagonal $y+z=1$ on the square $X=0$.

\begin{proposition}\label{Prop-3int}
Consider $m=\begin{pmatrix}0&\frac{1}{2}&\frac{1}{2}\end{pmatrix}$,
then $\mathcal L'_m$ has complexity function $2n+1, n\geq 0$, and $\mathcal L_m$ has complexity function $2n+3, n\geq 2$.
\end{proposition}
\begin{proof}

For the dashed circle:
$4-2\varphi=l_1+l_2, l_3=2\varphi-3$ for a total length of $1$.

By Lemma \ref{lem-nb-intervalles} we deduce $k(m)=3$. The three return words are $acb,abc, abb$. They all have the same length, thus it means that $a$ appears in $u$ with constant gap equal to $3$, and thus has frequency $1/3$. The word $\delta_au$ is a sturmian word with the same language as the Fibonacci word. Thus $u$ is obtained from the sturmian word $\delta_a(u)$ by $2$ by $2$ insertion of the letter $a$, see the work of \cite{Bar.Kab.17}.

The translation $R$ is coded by three intervals and all the lengths are in $\mathbb Q(\varphi)$: the discontinuity points are equal to $2-\varphi$ and $4-2\varphi$. The angle of the translation is $2\varphi-3$.
There is no saddle connection between the discontinuities, which means that there is no integer $q$ such that $2-\varphi+q(2\varphi-3)=4-2\varphi$. Thus the complexity of $\mathcal L'_m$ is equal to $2n+1$ for all $n\geq 0$ see Lemma \ref{lem-rot-complexite}. Now we need to use the recoding $\Phi$. There are two right special words in $\mathcal L'_m$ which have infinite left prefixes. They have for right extensions $a_1, a_2$ (resp. $a_2, a_4$). The images by $\Phi$ have $a$ and $ab$ for common suffixes. Thus we deduce that $\mathcal L_m$ has two right special words for each length. It remains to look at right special words of length one: $a,b,c$. We deduce $p(2)-p(1)=4, p(n+1)-p(n)=2$ for $n\geq 2$. Hence $p(n)=2n+3$ for $n\geq 2$.

\end{proof}

\subsection{Conclusion of the proof of Theorem \ref{thm2}}
Thus we have six cases:
Four cases where $k(m)=5$ and we use Proposition \ref{Prop-5int}, one case where $k(m)=3$ where we use Proposition \ref{Prop-3int} and one generic case where $k(m)=6$ where we use Proposition \ref{prop-6int}.

%%%%%%%%%%%%%%%%%%%%%%%%%%%%%%%%%%%%%%%
%%%%%%%%%%%%%%%%%%%%%%%%%%%%%%%%%%%%%%%%%

%%%%%%%%%%%%%%%%%%%%%%%%%%
%%%%%%%%%%%%%%%%%%%%%%%%%%%
\section{Proof of Theorem \ref{thm3}}

\begin{definition}
A generalized diagonal in the direction $\theta_0$ is a segment in $\mathbb R^3$ of direction $\theta_0$ which starts and ends on points of unit intervals of the one dimensional complex $\mathbb Z^3$. Such a trajectory is said to be irreducible if it does not intersect other edge of the complex in between. The length of such an irreducible segment is the number of cubes intersected minus one. 
\end{definition}

Consider a bispecial word $v$ of $\mathcal{L}(\theta)$, then its cell is a polygon inside one face. By Proposition 2 of \cite{Bed03} there is a one to one correspondance between the bispecial words and the irreducible generalized diagonals in direction $\theta$.  

Remark that $\frac{1}{\theta_3}=\frac{1}{\theta_1}+\frac{1}{2\theta_2}.$ By Lemma 40 of \cite{Bed.07} we deduce that $s(n)$ is equivalent to $Cn$ with $C=2-\frac{\theta_3}{\theta_1+\theta_2+\theta_3}=\frac{2(4+\varphi)}{6}$.

By Lemma \ref{julien}, the number of bispecial words is given by $s(n+1)-s(n)$. 
We conclude $p(n)\sim C\frac{n^2}{2}=\frac{4+\varphi}{6}n^2$.

%%%%%%%%%%%%%%%%%%%%%%%%%%
%%%%%%%%%%%%%%%%%%%%%%%%%%%%
\section{Remarks and comments}
In this last part we explain how we can change the direction $\theta_0$ and obtain with the same method some  similar results.
\subsection{First example}
We consider direction $\begin{pmatrix}r& \frac{1}{\varphi}& \frac{1}{\varphi^2}\end{pmatrix}$ with $r$ a non zero rational number. We explain that the same method works with small changes:

\begin{itemize}
\item The slopes of the lines in red and blue are the same. But the horizontal and vertical segments change as function of $r$. Of course the names of the words also.

\item The formula of Lemma \ref{lem-translation} becomes  $\Phi(m+k\frac{\theta_2}{r}\begin{pmatrix} 1& -1\end{pmatrix} \mod 1)$.

\item The number of elements of the partition of $I$ change with $r$.
\end{itemize}

With these changes, it will be easy and left to the reader to modify the statements of Theorem \ref{thm2}.

\subsection{More general example}
Consider a direction $\theta$ and the vector space $V_\theta$ over $\mathbb Q$ generated by $\theta_1, \theta_2, \theta_3$. There are three cases:
\begin{itemize}
\item $dim V_\theta=3$. Then there are two subcases depending if $1/\theta_1, 1/\theta_2, 1/\theta_3$ are linearly independant over $\mathbb Q$ or not. In all the cases $p(n,m,\theta)=p(n,\theta)\sim Cn^2$, $\mathcal L(m,\theta)=\mathcal L(\theta)$, a precize description of the language is still not known. We can refer to \cite{Andrieu-vivion} for some interesting properties on the balanceness property or to \cite{Berthe-21} for a survey.
\item $dim V_\theta=1$. Every trajectory is periodic, thus we have $p(n,m,\theta)\leq p(n,\theta)\leq C$.

\item $dim V_\theta=2$. This is the case treated in this paper.
Up to symetry such a direction has the form $\begin{pmatrix}a\theta_1+b\theta_2&\theta_1& \theta_2\end{pmatrix}$ with $a,b\in \mathbb Q$. Thus the billiard inside the square with direction $\begin{pmatrix}\theta_1& \theta_2\end{pmatrix}$ is coded by a sturmian word $v$. The goal is to find $u$ such that $\pi(u)=v$ and obtain its complexity.
We can find a partition similar to Figure \ref{fig-partition}. The number of elements in the partition can be estimated with the following method: it is enough to find for each $i$ the number $k_i$ such that $\frac{k_i-1}{\theta_i}<\frac{1}{\theta_3}<\frac{k_i+1}{\theta_i}$. It will give the number of letters $i$ in a return word. We will deduce an upper bound of the size of a return word, and thus an upper bound of the number of elements of the partition.
In our example we have only one $c$ and one or two $b$.
Then for each $m$ in the face $X=0$ we have a circle $I$ as in Lemma \ref{lem-nb-intervalles}. 
We could deduce with the same method the following result

\begin{theorem}
For a direction $\theta$ such that $dim V_\theta=2$, and for a point $m$ with a well defined orbit, we have $p(n,m,\theta)\leq Cn$ with $C$ function of $\theta$.
\end{theorem}

\end{itemize}

\subsection{Higher dimension}
We can ask the same question in an hypercube of $\mathbb R^d$ with $d\geq 4$. The generic case has been studied in \cite{Ba.95} see also \cite{Bed.09}. We can 
conjecture the following formula $p(n,m,\theta)=\Theta(n^{dim V_\theta-1})$.
For the global complexity $p(n)$, we refer to \cite{Bed.Hub.07} where we obtain the asymptotic.

\bigskip

\thanks{The authors would like to thanks {\sc Afrimath RT CNRS} for their helps during this work, and Project IZES ANR-22-CE40-0011.}

\bibliographystyle{abbrv}

\bibliography{biblio-bbc.bib}
\end{document}